\documentclass[11pt]{article}
\usepackage{amsmath,amssymb,algorithm,xcolor,amsthm,a4wide,dsfont,subcaption,graphicx}
\usepackage{algpseudocode}
\usepackage[framemethod=TikZ]{mdframed}
\newcounter{theo}[section]
\renewcommand{\thetheo}{\arabic{section}.\arabic{theo}}
\newenvironment{theo}[2][]{%
	\refstepcounter{theo}%
	\ifstrempty{#1}%
	{\mdfsetup{%
			frametitle={%
				\tikz[baseline=(current bounding box.east),outer sep=0pt]
				\node[anchor=east,rectangle,fill=blue!20]
				{\strut Theorem~\thetheo};}}
	}%
	{\mdfsetup{%
			frametitle={%
				\tikz[baseline=(current bounding box.east),outer sep=0pt]
				\node[anchor=east,rectangle,fill=blue!20]
				{\strut Theorem~\thetheo:~#1};}}%
	}%
	\mdfsetup{innertopmargin=10pt,linecolor=blue!20,%
		linewidth=2pt,topline=true,%
		frametitleaboveskip=\dimexpr-\ht\strutbox\relax
	}
	\begin{mdframed}[]\relax%
		\label{#2}}{\end{mdframed}}
\newcounter{lem}[section] 
\renewcommand{\thelem}{\arabic{section}.\arabic{lem}}
\newenvironment{lem}[2][]{%
	\refstepcounter{lem}%
	\ifstrempty{#1}%
	{\mdfsetup{%
			frametitle={%
				\tikz[baseline=(current bounding box.east),outer sep=0pt]
				\node[anchor=east,rectangle,fill=green!20]
				{\strut Lemma~\thelem};}}
	}%
	{\mdfsetup{%
			frametitle={%
				\tikz[baseline=(current bounding box.east),outer sep=0pt]
				\node[anchor=east,rectangle,fill=green!20]
				{\strut Lemma~\thelem:~#1};}}%
	}%
	\mdfsetup{innertopmargin=10pt,linecolor=green!20,%
		linewidth=2pt,topline=true,%
		frametitleaboveskip=\dimexpr-\ht\strutbox\relax
	}
	\begin{mdframed}[]\relax%
		\label{#2}}{\end{mdframed}
	}
\newcounter{corr}[section] 
\renewcommand{\thecorr}{\arabic{section}.\arabic{corr}}
\newenvironment{corr}[2][]{%
	\refstepcounter{corr}%
	\ifstrempty{#1}%
	{\mdfsetup{%
			frametitle={%
				\tikz[baseline=(current bounding box.east),outer sep=0pt]
				\node[anchor=east,rectangle,fill=yellow!20]
				{\strut Corollary~\thecorr};}}
	}%
	{\mdfsetup{%
			frametitle={%
				\tikz[baseline=(current bounding box.east),outer sep=0pt]
				\node[anchor=east,rectangle,fill=yellow!20]
				{\strut Corollary~\thecorr:~#1};}}%
	}%
	\mdfsetup{innertopmargin=10pt,linecolor=yellow!20,%
		linewidth=2pt,topline=true,%
		frametitleaboveskip=\dimexpr-\ht\strutbox\relax
	}
	\begin{mdframed}[]\relax%
		\label{#2}}{\end{mdframed}}

\newcounter{cond}[section] 
\renewcommand{\thecond}{\arabic{section}.\arabic{cond}}

\author{Sadok Jerad\thanks{Mathematical Institute, Woodstock Road, University of Oxford, Oxford, UK, OX2 6GG.  \texttt{sadok.jerad@maths.ox.ac.uk}}}
\title{Fast Adaptive Tensor Methods Under Local Smoothness}

\newtheorem{assumption}{Assumption}
\newcommand{\Ren}{\mathbb{R}^n}
\newcommand{\eqdef}{\stackrel{\rm def}{=}}

\newcommand{\iibe}[2]{\{ #1, \ldots, #2 \}}

\newcommand{\maxkvarsigma}{\max(\varsigma, \|g_k\|)}
\newcommand{\maxivarsigma}{\max(\varsigma, \|g_i\|)}

\newcommand{\minhalfgk}{\min(\frac{1}{2}, \frac{\|g_k\|}{2 \varsigma})}
\numberwithin{equation}{section}
\newcommand{\minhalfgi}{\min(\frac{1}{2}, \frac{\|g_i\|}{2 \varsigma})}
\newcommand{\kap}[1]{\kappa_{\mbox{\tiny #1}}}
\newcommand{\calS}{{ \mathcal{S}}} 
\newcommand{\calI}{{\mathcal{I}}}

\newcommand{\sfrac}[2]{{\scriptstyle \frac{#1}{#2}}}
\begin{document}
	\maketitle
	\begin{abstract}
A new, fast adaptive regularization methods is proposed and analyzed under local Lipschitz smoothness of the $p$-th order tensor. For nonconvex problems, it achieves the optimal $\mathcal{O}\!\left(|\log(\epsilon)|\epsilon^{-(p+1)/p}\right)$ complexity to obtain first-order $\epsilon$-stationary points and in the convex case, it yields $\mathcal{O}\!\left(|\log(\epsilon)|\epsilon^{-1/p}\right)$ iterations to drive the optimality gap below $\epsilon$, thus matching the complexity bounds of standard tensor methods under global Lipschitz smoothness yp to logarithmic terms. The proposed algorithm follows the line of standard tensor methods with an appropriately chosen regularization and suitable modifications. Initial numerical experiments and comparisons for some nonconvex regression problems are made with the standard adaptive cubic regularization where we showcase some potential of the proposed method.	
\end{abstract}
\section{Introduction}

	Classical optimization methods, such as gradient descent and Newton's method, have been successfully applied to a wide variety of practical problems arising in many scientific and engineering disciplines. Their theoretical analysis, however, typically relies on the global Lipschitz continuity of a particular derivative (usually the gradient or the Hessian)~\cite{Cartis2022-wb}. This assumption does not hold even for simple one-dimensional functions such as polynomials and the exponential function. Although the iterates of an optimization algorithm often remain inside a bounded level set (provided the objective function decreases monotonically), global Lipschitz continuity is then satisfied only locally, and the associated Lipschitz constant may depend on the initialization and become arbitrarily large.
		In the convex setting, several works have investigated optimization methods under only local gradient Lipschitz continuity~\cite{MalMischenko24,ZhouMa25}, or by introducing notions of smoothness based on Bregman divergences~\cite{Bauschke2017,Lu2018}. For the nonconvex case, motivated by empirical observations on the loss landscapes of neural networks, Zhang et al.~\cite{Zhang2020Why} proposed the following generalized smoothness assumption for gradient descent under both deterministic and stochastic settings:
\begin{equation}\label{L0L1base}
	\|\nabla_x^2 f(x)\| \leq L_0 + L_1 \|\nabla_x^1 f(x)\|,
	\end{equation}
where $L_0, L_1 > 0$.

This assumption encompasses a broader class of functions, including univariate polynomials and exponential functions; see \cite{Zhang2020Why,Vankovetal25} for further examples and discussion. It has subsequently been extended to various first-order optimization frameworks, including stochastic optimization, variational inequalities, variance-reduced methods, and different step-size strategies; see \cite{pmlr-v195-faw23a,pmlr-v206-sun23d,hubler24a,ZhangFang20,Chenetal23,ReiszeiJadb25,Koloskova23,Leietal23} and the references therein.
Observe also  that \eqref{L0L1base} may be rewritten by only using first-order derivative as

\begin{equation}\label{L0L1grad}
\|\nabla_x^1 f(x) - \nabla_x^1 f(y)\| \leq (L_0 + L_1 \|\nabla_x^1 f(x)\|) \|x-y\|, \quad \quad \text{ if } \|x-y\| \leq \frac{1}{L_1},
\end{equation} 
see \cite{Leietal23,Vankovetal25} for discussions and relations between both \eqref{L0L1base} and \eqref{L0L1grad}.

Extending this framework to higher-order methods is particularly appealing, as such methods enjoy improved worst-case evaluation complexity over first-order algorithms \cite{BirgGardMartSantToin17,Nesterov2019}, faster local convergence rates \cite{Doikov2021}, and, in certain applications, superior practical performance compared with Newton-type methods \cite{cartis2025efficient}. Recently, the works of \cite{Xieyinu24,GraJerToin25} investigated second-order methods under Hessian Lipschitz conditions analogous to \eqref{L0L1grad}. However, the first work in \cite{Xieyinu24} requires prior knowledge of Lipschitz smoothness constant making the method non-adpative in addition to solving a trust-region sub-problem exactly at each iteration. \cite{GraJerToin25} instead proposed a second-order algorithm alternating between regularized Newton steps and negative-curvature directions, (with the same spirit as the method in \cite{GraJerToin24}) under local Lipschitz smoothness of the Hessian. The method is parameter-free and achieves a complexity of $\mathcal{O}\!\left(|\log (\epsilon)| \epsilon^{-3/2}\right)$ to reach an $\epsilon$- first order stationary point. More recently \cite{Semenovetal25} proposed a new smoothness condition for regularized Newton methods by describing the \emph{ maximum radius of a ball around the current point
	that yields a good relative approximation of the gradient field.} This condition was later applied to unify the analysis of various gradient-regularized convex Newton methods  \cite{Mishchenko2023,Doikov2021,DoikovMischNes24} and to provide an analysis of Gauss-Newton methods for some machine learning problems. 

Motivated by the local Hessian smoothness assumptions introduced in \cite{Xieyinu24,GraJerToin25}, we develop a new local Lipschitz condition framework for $p$-th order tensor methods together with a corresponding adaptive regularization algorithm. The proposed method follows the philosophy of classical adaptive regularization algorithms \cite{BirgGardMartSantToin17}, but introduces a new regularization term that naturally reflects the proposed smoothness condition.  Because the smoothness assumption is only local, additional acceptance tests are required beyond the classical ratio test as done in \cite{GraJerToin25}. For nonconvex optimization, the resulting algorithm achieves an evaluation complexity of
 $\mathcal{O}\!\left(|\log(\epsilon)| \epsilon^{-(p+1)/p}\right)$ complexity to reach an $\epsilon$-first order stationary point, thus retrieving the optimal complexity of tensor methods \cite{BirgGardMartSantToin17} (under global Lipschitz assumption) up to a logarithmic factor. In the convex regime, we retrieve a $\mathcal{O}\left(\epsilon^{-1/p} |\log(\epsilon)|\right)$ complexity to drive the optimality gap below the threshold $\epsilon$, which again matches the corresponding complexity of standard tensor methods up to logarithmic terms.

The paper is organized as follows. Section~\ref{Section2} starts by stating our new Lipschitz smoothness condition for the $p$th-order tensor, describes the general algorithmic, compares it with standard tensor methods, specify the algorithm for $p=1$ and states other properties. Section 3 derives
a bound on its worst-case complexity for reaching $\epsilon$ threshold of criticality. Subsection~\ref{nonconvexanalysis} focus on the nonconvex case and  Subsection~\ref{convexanalysis} deals with the convex one. Results are specified in each time for the case of interest $p=2$. The numerical behaviour of
the proposed algorithms is considered in Section~\ref{numeric-s} and some conclusions and perspectives are
provided in Section~\ref{concl-s}.

\vspace*{1em}

\textbf{Notations} Let $n \geq 1$. The symbol $\|. \|$ denotes the Euclidean norm for vectors in IRn
and
its associated subordinate norm for matrices and tensors.  For two vectors $x, y \in \mathbb{R}^n$,  $x^\intercal y$ denotes their inner
product. $I_n$ is the identity matrix in $\mathbb{R}^{n \times n}$
.

\section{Adaptive Regularization Methods with Local Smoothness}\label{Section2}
We consider the problem of finding approximate first-order critical points of the smooth unconstrained optimization problem 
\begin{equation}\label{problemf}
	\min_{x \in \Ren} f(x),
\end{equation}
under the following set of assumptions.
\begin{assumption}\label{assum1}
	$f$ is p times continuously differentiable.
\end{assumption}
\begin{assumption}\label{assum2}
	There exists a constant $f_{\rm low}$ such that $f(x) \geq f_{\rm low}$ for all $x \in \Ren$.
\end{assumption}
\begin{assumption}\label{assum3}
	There exist constants $L_0 \geq 0$ and $L_1 \geq 0$ and $\delta > 0$, such that if $\|x-y\| \leq \delta$
	\begin{equation}\label{Lippth}
		\| \nabla^p_x f(x) - \nabla^p_x f(y) \| \leq (L_0 + L_1 \|\nabla_x^1 f(x)\|) \|x-y\|.
	\end{equation}
\end{assumption}
\begin{assumption}\label{assum4}
	There exist constant $\{M_{0,i}\}_{i=2}^p$ and $\{M_{1,i}\}_{i=2}^p$ such that 
	\begin{equation}\label{boundtensor}
		\max_{\|u\|=1} - \nabla_x^{i} f(x)[u]^i \leq M_{0,i} + M_{1,i} \|\nabla_x^1 f(x)\| \textrm{ for all } i \in \iibe{2}{p}.
	\end{equation}
\end{assumption}

Assumption~\ref{assum1} and Assumption~\ref{assum2} are standard when studying $p$-th-order method. Assumption~\ref{assum3} is a weaker assumption than the standard global Lipschitz continuity of the $p$-th tensor that it is usually assumed \cite{BirgGardMartSantToin17,Nesterov2019}. In addition to its local character, we have also added a term in $\|\nabla_x^1 f(x)\|$ to allow larger growth of the $p$-th order tensor differences with respect to $\|\nabla_x^1 f(x)\|$. Note that for $p=1$, \eqref{Lippth} reduces to the newly introduced  $(L_0, \, L_1)$ smoothness condition of the gradient \eqref{L0L1grad} discussed in the Introduction. Assumption~\ref{assum4} is related to the bound on second-to higher tensors terms that was employed in \cite{GraJerToin24,OFFO-ARp} to develop Objective Function-Free adaptive tensor methods. Note that we are less restrictive as we even allow the bound to scales with gradient magnitude.

As done in \cite{Xieyinu24}, we now propose two conditions easier to verify that imply Assumption~\ref{assum3}. Since extensive works has already been developed for $(p=1), \, (L_0, L_1)$ smoothness \cite{Zhang2020Why,Vankovetal25,Leietal23}, we focus on $p \geq 2$.

\begin{lem}{caractAS3}
Suppose $p \geq 2$ and let $f$ be $p+1$ times differentiable and suppose there exists $M_0 , G_0 \geq 0$ and $M_1, G_1 > 0$ such that 
\begin{align}
\|\nabla_x^{p+1} f(x)\| &\leq M_0 + M_1 \|\nabla_x^1 f(x)\|, \label{pplusone} \\
\|\nabla_x^2 f(x)\| &\leq G_0 + G_1 \|\nabla_x^1 f(x)\| \label{L0L1}.
\end{align}
Then Assumption~\ref{assum3} is satisfied with $L_0 = M_0 + \frac{G_0 M_1}{G_1}$, $L_1 = 2M_1 $ and $\delta = \frac{1}{G_1}$.  
\end{lem}
\begin{proof}
	As the proof is in the spirit of \cite[Lemma~C.1]{Xieyinu24}, we defer it to Appendix~\ref{firstmsproof}.
\end{proof}

As a consequence, it can be proved thanks to Lemma~2.1, that the class of functions satisfying Assumption~\ref{assum3} contains univariate polynomials of any degree and one dimensional exponentials. Note that all these functions satisfy also Assumption~\ref{assum4}. Indeed, for the aforementioned functions, the derivatives of higher degree grow (at infinity) slower than the first derivative. 

We now derove the impact of this new Lipschitz condition on the standard $p$-th order approximation of the function value and the gradient \cite{Cartis2022-wb}.
\begin{lem}{approxfungrad}
	Suppose that Assumption~\ref{assum1} and Assumption~\ref{assum3} hold. Let $x$ and $s$ such that $\|s\| \leq \delta$. Then, we have that
	\begin{equation}\label{Lipf}
		|f(x+s) - T_{f,p}(x,s)| \leq \frac{L_0 + L_1 \|\nabla_x^1 f(x)\| }{(p+1)!} \|s\|^{p+1},
	\end{equation}
	and 
	\begin{equation}\label{Lipg}
		\|\nabla_x^1 f(x+s) - \nabla_s^1 T_{f,p}(x,s)\| \leq \frac{L_0 + L_1 \|\nabla_x^1 f(x)\| }{p!} \|s\|^{p}.
	\end{equation}
\end{lem}
\begin{proof}
	As the proof is close to standard $p$-th order Lipschitz error  \cite[Lemma~2.1]{CartGoulToin20b}, it is deferred to the Appendix~\ref{approxfungrad}.
\end{proof}

%Remark that \eqref{assum4} is weaker than imposing the same bound with  $\|\nabla_x^i f(x)\|$ in the left hand-side.
As standard in $p$-th-order approximation, we denote the $p$-th order Taylor approximation as 
\begin{equation}\label{Taylorapprox}
	T_{f,p}(x,s) \eqdef f(x) + \sum_{i=1}^{p} \frac{\nabla_x^i f(x)[s]^i}{i!}.
\end{equation}

Before stating our algorithm, we briefly review the standard adaptive regularization framework \cite{BirgGardMartSantToin17,Nesterov2019}. For $p$-th-order globally continuous Lipschitz function, the method builds at each method a regularized Taylor model which writes as 
\begin{equation}\label{standardmodel}
	T_{f,p}(x_k,s_k) + \frac{\sigma_{k} \|s_k\|^{p+1}}{(p+1)!}. 
\end{equation}
However this model only works for globally-Lipschitz continuous $p$-th order function, and a decrease is ensured whenever $\sigma_{k}$ is larger than the global Lipschitz constant. In our case, we  adjust  the model's definition to account for the new smoothness condition by incorporating first-order information in the regularization. Our model therefore writes as 
\begin{equation}\label{model}
	m_k(s_k) = T_{f,p}(x_k,s_k) + \frac{\sigma_{k} \max(\varsigma, \|\nabla_x^1 f(x_k)\|) \|s_k\|^{p+1}}{(p+1)!},
\end{equation}

where $ \varsigma > 0$ and $T_{f,p}$ defined in \eqref{Taylorapprox}. Observe that the current model \eqref{model}  is more conservative than \eqref{standardmodel} because it imposes stronger regularization where the norm of the gradient is large. Intuitively, it is in accordance with the local smoothness condition \eqref{Lippth} where the gradient intervenes in the left hand side of the bound.

\noindent
We are now in a position to state our \texttt{ARp-LS} algorithm. 

\begin{algorithm}
	\caption{Adaptive $p$-th order regularization with Local Smoothness \texttt{ARp-LS}}\label{ARp-LS}
	\begin{algorithmic}[1]
		\Require An initial point $x_0 \in \Ren$, $\sigma_0 > 0$, $\epsilon \in (0,1]$  are given, as well as the parameters
		\[
		\sigma_{\min}  > 0, \,  \, \vartheta  \geq 1, \, 1 \geq \varsigma > 0, \kappa_\theta \geq 1, \, \theta_1 \geq 1 
		\]
		\[
		0< \gamma_1 < 1 < \gamma_2 \leq \gamma_3 
		\textrm{ and }
		0 < \eta_1 \leq \eta_2 < 1.  
		\]
		Define also  $\kap{sup} \eqdef \frac{\kappa_\theta + \theta_1}{p!}$, $\kap{upgrad} \eqdef \frac{\vartheta + \theta_1}{p!}$ and \texttt{REJECT = FALSE}. 
		\State $k \gets 0$ and compute $g_0 = \nabla_x^1 f(x_0)$.
		 \While{$\|g_k\| \geq \epsilon$}
		\State Compute the true derivatives $\nabla^i_x f(x_k)$ for $i \in \iibe{2}{p}$.
		\State Compute an approximate minimizer of the model defined in \eqref{model} in the sense that
		\begin{equation}\label{minimizmodel}
			m_k(s_k) - m_k(0) < 0, \quad \|\nabla_s^1 T_{f,p}(x_k,s_k)\| \leq \theta_1 \frac{\maxkvarsigma \|s_k\|^p}{p!}
		\end{equation}
		\State If 
		\begin{equation}\label{teststeplength}
			\|\nabla_x^1 f(x_k + s_k)\| > \frac{\|g_k\|}{2}, \quad \textrm{ and } \quad \kap{sup} \sigma_k \|s_k\|^p < \min(\frac{1}{2}, \frac{\|g_k\|}{2 \varsigma}),  
		\end{equation}
		set \texttt{REJECT = TRUE} and proceed to Line~\ref{updatestep}. \label{Line5}
		\State If 
		\begin{equation}\label{testnextgrad}
			\|\nabla_x^1 f(x_k + s_k)\| > \kap{upgrad} \maxkvarsigma \sigma_k \|s_k\|^p, 
		\end{equation}
		set \texttt{REJECT = TRUE} and proceed to Line~\ref{updatestep}. \label{Line6}
		\State 
		Evaluate $f(x_k +s_k)$ and compute the acceptance ratio
		\begin{equation}\label{rhokdef}
				\rho_k = \frac{f(x_k) - f(x_k +s_k)}{f(x_k) - T_{f,p}(x_k,s_k)}.
		\end{equation}
		If $\rho_k < \eta_1$, set \texttt{REJECT = TRUE}.
		
		\State If \texttt{REJECT = FALSE}, set $x_{k+1} = x_k + s_k$, otherwise set $x_{k+1} = x_k$.                     \label{updatestep}
		\State Set
		\begin{equation}\label{sigmakupdate}
			\sigma_{k+1} \in \left\{
			\begin{aligned}
				& [ \max\left(\sigma_{\min}, \gamma_1\sigma_k \right),\sigma_k] 
				&&\text{ if } \rho_k \geq \eta_2 
				&&\text{and } \texttt{REJECT=FALSE}, \\
				& [\sigma_k ,\gamma_2\sigma_k ]            
				&&\text{ if } \rho_k \in [\eta_1, \eta_2) &&\text{and } \texttt{REJECT=FALSE},\\
				& [\gamma_2 \sigma_{k}, \gamma_3 \sigma_{k}] 
				&&\text{ if } \rho_k < \eta_1 
				&&\text{and } \texttt{REJECT=TRUE}, 
			\end{aligned}
			\right.
		\end{equation}
		Increment $k$ by one, set \texttt{REJECT = FALSE} and compute $g_k = \nabla_x^1 f(x_k)$. 
		\EndWhile
	\end{algorithmic}
\end{algorithm}
As our algorithm differs from standard tensor methods \cite{BirgGardMartSantToin17,Cartis2022-wb,Nesterov2019}, we provide some explanation on the algorithm. 
In addition to the standard test of the sufficient decrease on the function value \eqref{rhokdef}, we add two additional tests. These two tests are required since the Lipschitz assumption is now local. The first tests ensures that the step ensures progress in two different senses. Either by dividing the norm of the gradient  at the next iterate (first part of \eqref{teststeplength}) by half or by obtaining a step that is sufficiently large (second part of \eqref{teststeplength}). We have also to ensure that no blow-up occurs at the next iterate and that the gradient at the next iterate remains bounded \eqref{testnextgrad}, a new theoretical condition that has been considered for regularized Newton method \cite{GraJerToin24,Mishchenko2023}. Note that these tests are not new and that they have also been considered in \cite{GraJerToin25} for a fast second-order method under local Lipschitz-smoothness.  For each of the three acceptance conditions imposed on the step, we will show in Section~\ref{Complexity-s} that for sufficiently large $\sigma_k$, all three tests (\eqref{testnextgrad}, \eqref{updatestep}, \eqref{rhokdef}) will be passed and \texttt{REJECT} remains \texttt{FALSE} and the iterate will be updated accordingly.  

Observe also that for $p=1$ and by solving exactly \eqref{model}, the step writes as
\begin{equation}\label{clippedgradient}
s_k = -\frac{\min(\frac{1}{\varsigma}, \frac{1}{\|g_k\|})}{\sigma_k} g_k
\end{equation}  
which is the clipped gradient descent \cite{Zhang2020Why,Koloskova23} which was shown to be optimal under the $(L_0, L_1)$ smoothness. Note that in our case, we propose a parameter-free framework as done in \cite{hubler24a}. In the last reference, a complexity of $\mathcal{O}\left(\epsilon^{-2}\right)$ was obtained to retrieve an $\epsilon$-first order stationary point.

Regarding the approximate minimization of the model in \eqref{minimizmodel}, the required conditions  follow the proposal of \cite{GratToin21} and they  extend the more usual conditions where the step $s_k$ is chosen as 
\[
\|\nabla_s^1 m_k(s_k)\| \leq \theta_1 \|s_k\|^p.
\]
Indeed, it is easy to verify that  \eqref{minimizmodel} holds on a local minimizer of $m_k$ with $\theta_1 > 1$. 

Following the standard notation
 of adaptive regularization methods \cite{BirgGardMartSantToin17} and taking into account the new conditions imposed in order to accept the step, we define the following notations

\[
\calS \eqdef \{ k \geq 0 \mid x_{k+1} = x_k+s_k \}, 
\]
the set of indexes of ``successful iterations'', and
\[
\calS_k \eqdef \calS \cap \iibe{0}{k}.
\]
Moreover, considering \eqref{teststeplength}, we further the iterations into two different subsets
\begin{align}
	\calI^{g\searrow} &
	\eqdef \left\{i \geq 0  \, \Big| \; \|\nabla_x^1f(x_i+s_i)\| \leq \frac{\|g_i\|}{2} \, \, \right\}, 
	\quad
	\calI^{g\nearrow} \eqdef \mathbb{N} \setminus \calI^{g\searrow}, \label{calIipdown} \\
	%\calI^{decr} &=  \left\{i\in \calI^{newt} \, | \,  \|\nabla_x^1f(x_i+s_i^{trial})\| > \frac{\|g_i\|}{2} \tim{and} \, \|s_i^{trial}\| \geq \frac{1}{\sqrt{\sigma_{i}}\kap{slow}} \right\}, \\ 	
	\calI^{decr} &\eqdef  \left\{i\in \calI^{g\nearrow} \, \Big|  \; \kap{sup} \sigma_{i} \|s_i\|^p \geq  \min(\frac{1}{2}, \frac{\|g_i\|}{2 \varsigma}) \right\} \label{calIdecr},
\end{align} 
the last subset containing the indices of iterations where both conditions in \eqref{teststeplength} fail.
The corresponding subsets of successful iterations are then given by
\[
	\calS_k^{g\searrow} \eqdef \calS_k \cap \calI^{g\searrow}, \quad \quad
\calS_k^{decr} \eqdef \calS_k \cap \calI^{decr}.
\]
Since the iteration is unsuccessful if the test \eqref{teststeplength} holds, one checks that
\begin{equation}\label{Skdivision}
\calS_k \eqdef \calS_k^{g\searrow} \cup \calS_k^{decr}.
\end{equation}

We also recall a well-known result bounding the total number of
iterations of adaptive regularization methods in
terms of the number of successful ones.

\begin{lem}{SvsU}
Suppose that the \texttt{ARp-LS} algorithm is used and that $\sigma_k \leq
\sigma_{\max}$ for some $\sigma_{\max} >0$. Then
\[
k \leq |\mathcal{S}_k| \left(1+\frac{|\log\gamma_1|}{\log\gamma_2}\right)+
\frac{1}{\log\gamma_2}\log\left(\frac{\sigma_{\max}}{\sigma_0}\right).
\]
\end{lem}
\begin{proof}
	See  	\cite[Theorem~2.4]{BirgGardMartSantToin17} or \cite[Lemma~2.4.1]{Cartis2022-wb}.
\end{proof}

This result implies that the overall complexity of the algorithm can be
estimated once bounds on $\sigma_k$ and $|\mathcal{S}_k|$ are known, as we will show in the
next section.

We now derive  upper bounds on the stepsize for all iterations. Although the proof follows the classical analysis of adaptive regularization methods \cite{OFFO-ARp,CartGoulToin19}, we keep the proof in the main text as it shows in detail the impact of both the new tensors conditions \eqref{assum4}  and the new choice of the regularization term \eqref{model}.

\begin{lem}{skbound}
	Suppose that Assumption~\ref{assum1} and Assumption~\ref{assum4} hold. Then, for all $k \geq 0$,
	\begin{equation}\label{exprskbound}
		\|s_k\| \leq \frac{\mu}{\sigma_{k}^\sfrac{1}{p}} 
	\end{equation}
	with $\mu$ defined as
	\begin{equation}\label{mudef}
		\mu \eqdef 2\max \left((p+1)!^\sfrac{1}{p}, \left\{ \sigma_{\min}^\sfrac{-i+1}{p(p-i+1)} \left(\frac{(\frac{M_{0,i}}{\varsigma} + M_{1,i}) (p+1)!}{i!}\right)^\frac{1}{p-i+1}\right\}_{i=2}^p\right).
	\end{equation}
\end{lem} 
\begin{proof}
	Rearranging the fact that $m_k(s_k) < m_k(0)$ and using \eqref{boundtensor}, we derive that
	\begin{align*}
		\frac{\sigma_{k} \max(\varsigma, \|g_k\|) \|s_k\|^{p+1}}{(p+1)!} &\leq -g_k^\intercal s_k + \sum_{i=2}^p - \frac{\nabla_x^i f(x_k)[s_k]^i}{i!}\\
		&\leq \|g_k\| \|s_k\| + \sum_{i=2}^p \frac{M_{0,i} + M_{1,i} \|g_k\|}{i!} \|s_k\|^i.
	\end{align*}
Applying now the Lagrange bound for polynomial roots \cite[Lecture~VI, Lemma~5]{Yap99} with $x=\|s_k\|^{p+1}$, $n=p+1$, $a_0=0$, $a_1 = -\|g_k\|$, $a_i = \frac{-M_{0,i} - M_{1,i} \|g_k\|}{i!}$ for $i \in \iibe{2}{p}$, and $a_{p+1} = \frac{\sigma_{k} \maxkvarsigma}{(p+1)!}$, it follows that  the equation $\sum_{i=0}^{n} a_i x_i$ admits at least one positive root, and we derive
\[
\|s_k\| \leq 2\max\left( \left(\frac{\|g_k\|(p+1)!}{\sigma_{k} \maxkvarsigma}\right)^\sfrac{1}{p},  \left\{ \left(\frac{(M_{0,i} + M_{1,i} \|g_k\|) (p+1)!}{i! \sigma_{k} \maxkvarsigma}\right)^\frac{1}{p-i}\right\}_{i=2}^p        \right).
\]
Now factorizing by $\frac{1}{\sigma_{k}^\sfrac{1}{p}}$ in the left hand side, using that $\frac{\|g_k\|}{\maxkvarsigma} \leq 1$ and that $\frac{(M_{0,i} + M_{1,i} \|g_k\|) }{\maxkvarsigma} \leq \frac{M_{0,i}}{\varsigma} + M_{1,i}$, we derive that 
\[
\|s_k\| \leq \frac{2}{\sigma_{k}^\sfrac{1}{p}} \max\left( (p+1)!^\sfrac{1}{p},  \left\{ \sigma_{k}^\sfrac{-i+1}{p(p-i+1)} \left(\frac{(\frac{M_{0,i}}{\varsigma} + M_{1,i}) (p+1)!}{i!}\right)^\frac{1}{p-i+1}\right\}_{i=2}^p        \right).
\]
Now using that  $-i+1 < 0$ for $i \in \iibe{2}{p}$ and that $\sigma_{k} \geq \sigma_{\min}$ from \eqref{sigmakupdate}, we derive \eqref{exprskbound} with the appropriate $\mu$ \eqref{mudef}. 
\end{proof}

Lemma~\ref{skbound} is crucial for our \texttt{ARp-LS} for $p$-th-order with local smoothness, as it shows that if $\sigma_{k} \geq (\frac{\mu}{\delta})^p$, \eqref{exprskbound} implies that $\|s_k\| \leq \delta$ and so the Lipschitz error bounds stated in \eqref{Lipf} and \eqref{Lipg} apply.

We next establish a standard bound on the lower bound for the decrease of the function values for successful iteration, a standard result in adaptive regularization methods \cite{BirgGardMartSantToin17}.

\begin{lem}{lowdecr}
Suppose that Assumption~\ref{assum1} holds and let $k \geq 0$. Then, we have that
\begin{equation}\label{locdecr}
	f(x_k) - T_{f,p}(x_k,s_k) \geq \frac{\sigma_{k} \maxkvarsigma \|s_k\|^{p+1}}{(p+1)!}.
\end{equation}
\end{lem}
\begin{proof}
	This is a consequence of the first inequality in \eqref{minimizmodel}.
\end{proof}

Before proceeding with our analysis, we prove a Lemma that upper-bounds $	\|\nabla_x^1 f(x_k + s_k)\|$, provided that the step is sufficiently small for \eqref{Lipg} to apply. 

\begin{lem}{gkplusonebound}
Let $k \geq 0$, suppose that Assumption~\ref{assum1} and Assumption~\ref{assum3} hold and that $\|s_k\| \leq \delta$. Then, we have that
\begin{equation}\label{exprgkplusonebound}
	\|\nabla_x^1 f(x_k + s_k)\| \leq \frac{\maxkvarsigma \sigma_{k} \|s_k\|^p}{p!} \left(\frac{L_0 + L_1 \|g_k\|}{\sigma_{k} \maxkvarsigma} + \theta_1\right).
\end{equation}

\end{lem}
\begin{proof}
Remark that since $\|s_k\| \leq \delta$, \eqref{Lipg} applies. Using the latter, with the second inequality of \eqref{minimizmodel}, we derive that
\begin{align*}
	\|\nabla_x^1 f(x_k + s_k)\| &\leq \|\nabla_x^1 f(x_k + s_k) - \nabla_s^1 T_{f,p}(x_k,s_k)\| + \|\nabla_s^1T_{f,p}(x_k,s_k)\| \\ 
	&\leq \frac{L_0 + L_1 \|g_k\|}{p!} \|s_k\|^p + \frac{\theta_1 \sigma_{k} \maxkvarsigma \|s_k\|^p}{p!}.
\end{align*} 
Factorizing in the previous inequality yields the desired result.
\end{proof}
We now turn to the complexity analysis, it will be divided into two cases. A first one for the generic non-convex case and a second  tailored for the convex one.

\section{Complexity analysis}\label{Complexity-s}
\subsection{Non-convex case} \label{nonconvexanalysis}
In the remainder of this subsection, we will proof two lemmas that show for sufficiently large $\sigma_{k}$, the algorithm bypasses
 both tests \eqref{teststeplength} and \eqref{testnextgrad} and proceeds to compute the standard acceptance ratio.
\begin{lem}{firsttest}
Suppose that Assumption~\ref{assum1} and Assumption~\ref{assum3} hold and that $\|s_k\| \leq \delta$. Then if 
\begin{equation}\label{sigmakfirstcond}
	\sigma_{k} \geq \frac{\frac{L_0}{\varsigma} + L_1}{\kappa_\theta},
\end{equation}
then \texttt{REJECT} remains \texttt{FALSE} after Line~\ref{Line5}.
\end{lem} 
\begin{proof}
For $k \in 	\calI^{g\searrow}$, we have that $\|\nabla_x^1 f(x_k +s_k)\| \leq \frac{\|g_k\|}{2}$ from \eqref{calIipdown} and so the first inequality in \eqref{teststeplength} is not true for any value of $\sigma_{k}$. We now consider $k \in 	\calI^{g\nearrow}$ and therefore $\frac{\|g_k\|}{2} \leq \|\nabla_x^1 f(x_k+s_k)\|$. Using Lemma~\ref{gkplusonebound} since $\|s_k\| \leq \delta$, we derive
\[
\frac{\|g_k\|}{2} \leq \|\nabla_x^1 f(x_k+s_k)\| \leq \frac{\maxkvarsigma \sigma_{k} \|s_k\|^p}{p!} \left(\frac{L_0 + L_1 \|g_k\|}{\sigma_{k} \maxkvarsigma} + \theta_1\right).
\]
Dividing by $\maxkvarsigma$ the l.h.s, using that $L_0 + L_1 \|g_k\| \leq (\frac{L_0}{\varsigma} + L_1) \maxkvarsigma$  and  the upper-bound on $\sigma_{k}$ \eqref{sigmakfirstcond}, we derive that
\[
\minhalfgk = \frac{\|g_k\|}{2\maxkvarsigma} \leq \frac{\sigma_{k} \|s_k\|^p}{p!} (\kappa_\theta + \theta_1) = \kap{sup} \sigma_{k} \|s_k\|^p,
\] 
where $\kap{sup}$ is defined in the initialization of \texttt{ARp-LS}. From the previous inequality, the second inequality of \eqref{teststeplength} cannot hold and so \texttt{REJECT} remains false.
\end{proof}

We next consider the test \eqref{testnextgrad} and we similarly prove that for sufficiently large $\sigma_{k}$, Algorithm~\ref{ARp-LS} continues beyond Line~\ref{Line6}.

\begin{lem}{sectest}
Suppose that Assumption~\ref{assum1} and Assumption~\ref{assum3} hold and that $\|s_k\| \leq \delta$. If 
\begin{equation}\label{sigmaksecondcond}
	\sigma_{k} \geq \frac{\frac{L_0}{\varsigma} + L_1}{\vartheta}, 
\end{equation}	
then \texttt{REJECT} stay \texttt{FALSE} after Line~\ref{Line6}.
\end{lem}
\begin{proof}
Using Lemma~\ref{gkplusonebound} since $\|s_k\| \leq \delta$, we derive 
\[
	\|\nabla_x^1 f(x_k + s_k)\| \leq \frac{\maxkvarsigma \sigma_{k} \|s_k\|^p}{p!} \left(\frac{L_0 + L_1 \|g_k\|}{\sigma_{k} \maxkvarsigma} + \theta_1\right).
\]
Now using \eqref{sigmaksecondcond} with $L_0 + L_1 \|g_k\| \leq (\frac{L_0}{\varsigma} + L_1) \maxkvarsigma$ in the last inequality,
\[
	\|\nabla_x^1 f(x_k + s_k)\| \leq \frac{\maxkvarsigma \sigma_{k} \|s_k\|^p}{p!} \left(\vartheta + \theta_1\right).
\]
  The definition of $\kap{upgrad}$  in the initialization of \texttt{ARp-LS} with the last inequality implies that \eqref{testnextgrad} is not valid and \texttt{REJECT} stays false and the algorithm proceeds beyond Line~\ref{Line6}. 
\end{proof}

Thanks to the two previous Lemmas~(\ref{sigmakfirstcond}, \ref{sigmaksecondcond}) and Lemma~\ref{skbound}, we are now ready to develop a bound on $\sigma_{k}$ for all iterations $k$.

\begin{lem}{sigmakbound}
Suppose that Assumption~\ref{assum1}, Assumption~\ref{assum3} and Assumption~\ref{assum4} hold and let $k \geq 0$. Then, we have that 
\begin{equation}\label{exprsigmakbound}
	\sigma_{k} \leq  \sigma_{\max} \eqdef\gamma_3 \max\left(\sigma_0, \, \frac{\frac{L_0}{\varsigma} + L_1}{\vartheta}, \frac{\frac{L_0}{\varsigma} + L_1}{\kappa_\theta},\, \frac{\mu^p}{\delta^p}, \, \frac{\frac{L_0}{\varsigma} + L_1}{(1-\eta_2)} \right),
\end{equation}
where $\mu$ is defined in \eqref{mudef}.
\end{lem}
\begin{proof}
First suppose that $\sigma_{k} \geq (\frac{\mu}{\delta})^p$ so that $\|s_k\| \leq \delta$ from \eqref{exprskbound} and the results of Lemma~\ref{approxfungrad} applies. Using the definition of \eqref{rhokdef}, Lipschitz error bound \eqref{Lipf} and  \eqref{locdecr}, we derive that,
\begin{equation}\label{rhokbound}
	1-\rho_k = \frac{f(x_k + s_k) - T_{f,p}(x_k,s_k)}{f(x_k) - T_{f,p}(x_k,s_k)} \leq \frac{(L_0 + L_1 \|g_k\|) \|s_k\|^{p+1}}{\sigma_{k} \maxkvarsigma \|s_k\|^{p+1}} = \frac{(L_0 + L_1 \|g_k\|)}{\sigma_{k} \maxkvarsigma }.
\end{equation}
And thus if
\[
\sigma_{k} \geq \frac{\frac{L_0}{\varsigma} + L_1}{1-\eta_2},
\]
plugging the last bound  in \eqref{rhokbound} yields that $\rho_k \geq \eta_2$ and so the step would be successful. Now using that $\|s_k\| \leq \delta$ so the results of both Lemma~\ref{firsttest} and Lemma~\ref{sectest} apply, we obtain that the step is successful provided that $\sigma_{k} \geq \max\left(\frac{\frac{L_0}{\varsigma} + L_1}{\vartheta}, \frac{\frac{L_0}{\varsigma} + L_1}{\kappa_\theta}\, \frac{\mu^p}{\delta^p}, \, \frac{\frac{L_0}{\varsigma} + L_1}{(1-\eta_2)}\right)$. At last, the mechanism of the algorithm \eqref{sigmakupdate} ensures that \eqref{exprsigmakbound} applies.
\end{proof}

We now provide a lemma that upper-bounds  $|\calS_k^{g\searrow}|$ w.r.t $|\calS_k^{decr}|$. The Lemma is in spirit of \cite[Lemma~3.4]{GraJerToin24}. This Lemma highlights the importance of \eqref{testnextgrad} as it upper-bounds $\frac{\|g_{k+1}\|}{\|g_k\|}$ even when the local Lipschitz error bounds does not hold.

\begin{lem}{Skgbound}
Suppose that Assumption~\ref{assum1}, Assumption~\ref{assum3} and Assumption~\ref{assum4} hold. Then, we have that 
\begin{equation}\label{Skgboundexpr}
|\calS_k^{g\searrow}| \leq \log(\frac{\kap{up}}{\epsilon}) \frac{|\calS_k^{decr}|}{\log(2)} + \frac{|\log(\epsilon)| + \log(\|g_0\|)}{\log(2)} + 1,
\end{equation} 
where $\kap{up}$ is defined as
\begin{equation}\label{kapupdef}
	\kap{up} =  \kap{upgrad}  \mu^p
\end{equation}
with $\kap{upgrad}$ defined in the initialization of \texttt{ARp-LS} and $\mu$ as in \eqref{mudef}.
\end{lem}
\begin{proof}
First observe that if $k \in \calS_k^{g \searrow}$, $\|g_{k+1}\| \leq \frac{\|g_k\|}{2}$. Let $k \in \calS_k^{decr}$. Using now that \eqref{testnextgrad} does not hold, that $\epsilon \leq \|g_k\|$, $\max(\varsigma, \epsilon) \leq 1$, and that \eqref{exprskbound} applies, we derive that,
\begin{equation}\label{gkplusovergk}
	\frac{\|g_{k+1}\|}{\|g_k\|} \leq \kap{upgrad} \max(\frac{\varsigma}{\|g_k\|}, 1) \sigma_{k} \|s_k\|^{p} \leq \kap{upgrad} \max(\frac{\varsigma}{\epsilon}, 1) \sigma_{k} \|s_k\|^{p} \leq \frac{\kap{upgrad}  \mu^{p}}{\epsilon} = \frac{\kap{up}}{\epsilon},
\end{equation}
where $\kap{up}$ is defined in \eqref{kapupdef}.
Successively using that $\calS_k = \calS_k^{decr} \cap  \calS_k^{g \searrow}$, the bound on $\frac{\|g_{k+1}\|}{\|g_k\|}$ in both cases either $i \in \calS_k^{decr}$ \eqref{gkplusovergk} or $i \in \calS_k^{g \searrow}$, we obtain that
\begin{align*}
	\frac{\epsilon}{\|g_0\|} &\leq \frac{\|g_k\|}{\|g_0\|} = \prod_{i \in \calS_k \setminus \{k\}} \frac{\|g_{i+1}\|}{\|g_i\|} = \prod_{i \in \calS_k^{decr} \setminus \{k\} } \frac{\|g_{i+1}\|}{\|g_i\|} \prod_{i \in \calS_k^{g\searrow} \setminus \{k\}} \frac{\|g_{i+1}\|}{\|g_i\|} \\
	&\leq \left(\frac{\kap{up}}{\epsilon}\right)^{\calS_k^{decr} \setminus \{k\}} \left(\frac{1}{2}\right)^{ \calS_k^{g\searrow} \setminus \{k\}}.
\end{align*}
Rearranging the last inequality and using that $|\calS_k^{decr} \setminus \{k\}| \leq |\calS_k^{decr}|$ yields that
\[
\frac{2^{ \calS_k^{g\searrow} \setminus \{k\} } \epsilon}{\|g_0\|} \leq \left(\frac{\kap{up}}{\epsilon}\right)^{\calS_k^{decr}}.
\]
Taking the logarithm in the last inequality, using $|\calS_k^{g\searrow} \setminus \{k\}| \geq |\calS_k^{g\searrow}| - 1$ and further rearranging yields the stated result \eqref{Skgboundexpr}.
\end{proof}
Equipped with the two last results (Lemma~\ref{sigmakbound} and Lemma~\ref{Skgbound}) and Lemma~\ref{SvsU}, we are now ready to state the complexity of Algorithm~\ref{ARp-LS}.

\begin{theo}{complexs}
	Suppose that Assumption~\ref{assum1}--\ref{assum4} hold. Then the \texttt{ARp-LS} algorithm requires at most
	\[
	|\calS_k| \leq \left(1 + \frac{\log(\kap{up}) |\log(\epsilon)|}{\log(2)} \right) \kappa_\star\epsilon^{-(p+1)/p} + \frac{|\log(\epsilon)| + \log(\|g_0\|)}{\log(2)} + 1
	\]
successful iterations and at most 
\begin{align*}
\left(1 + \frac{|\log (\gamma_1)|}{\log (\gamma_2)}\right) &\left[ \left(1 + \frac{\log(\kap{up}) |\log(\epsilon)|}{\log(2)} \right) \kappa_\star \epsilon^{-(p+1)/p} + \frac{|\log(\epsilon)| + \log(\|g_0\|)}{\log(2)} + 1\right] \\ 
&+ \frac{1}{\log \gamma_2} \log\left(\frac{\sigma_{\max}}{\sigma_0}\right)
\end{align*}
iterations to produce a vector $x_\epsilon$ such that $\|\nabla_x^1 f(x_\epsilon)\| \leq \epsilon$, where $\kappa_\star$ is defined as
\begin{equation}\label{kapstardef}
\kappa_\star \eqdef  \frac{(p+1)!(2\kap{sup})^\sfrac{p+1}{p}\sigma_{\max}^\sfrac{1}{p} (f(x_0) - f_{\rm low})}{\eta_1 \varsigma },
\end{equation}
with $\sigma_{\max}$ as in \eqref{exprsigmakbound}, $\kap{up}$ as in \eqref{kapupdef} and $\kap{sup}$ as defined  in the initialization of \texttt{ARp-LS}.
\end{theo}
\begin{proof}
Let us first start by deriving a bound $|\calS_k|$. We begin by bounding $|\calS_k^{decr}|$. Let $i \in \calS_k^{decr}$. Using \eqref{rhokdef}, \eqref{locdecr} and that the second part of \eqref{teststeplength} does not apply, we derive that
\begin{align}
	f(x_{i}) - f(x_{i+1}) &\geq \eta_1 (f(x_i) - T_{f,p}(x_i,s_i)) \geq \eta_1 \frac{\sigma_{i} \maxivarsigma \|s_i\|^{p+1}}{(p+1)!} \nonumber \\
	&\geq \eta_1\frac{ \varsigma \sigma_{i} 
	 \|s_i\|^{p+1}}{(p+1)!}
	 \geq  \eta_1\frac{\varsigma \minhalfgi^\sfrac{p+1}{p}}{(p+1)!\sigma_{i}^\sfrac{1}{p} \kap{sup}^\sfrac{p+1}{p}}
	 \label{toreuse} \\
	 &\geq \frac{\eta_1 \varsigma \epsilon^\sfrac{p+1}{p}}{(p+1)!(2\kap{sup})^\sfrac{p+1}{p}\sigma_{\max}^\sfrac{1}{p}} \nonumber
\end{align} 
where we used \eqref{exprsigmakbound}, that $\varsigma , \epsilon \leq 1$ and that $\|g_i\| \geq \epsilon$ before termination. Summing the above inequality for all $i \in \calS_k^{decr}$, using that $f(x_i)$ is a non-increasing sequence and Assumption~\ref{assum2}, we derive,
\[
f(x_0) - f_{\rm low} \geq \sum_{i=0}^{k} f(x_i) - f(x_{i+1}) \geq \sum_{i \in \calS_k^{decr}} f(x_i) - f(x_{i+1}) \geq  \frac{|\calS_k^{decr}|\eta_1 \varsigma \epsilon^\sfrac{p+1}{p}}{(p+1)!(2\kap{sup})^\sfrac{p+1}{p}\sigma_{\max}^\sfrac{1}{p}}.
\]
Using now the  definition of $\kappa_\star$ in \eqref{kapstardef}, rearranging the last inequality gives $|\calS_k^{decr}| \leq \kappa_\star \epsilon^{-(p+1)/p}$. Finally, since $|\calS_k| = |\calS_k^{decr}| + |\calS_k^{g\searrow}|$ and the bound \eqref{Skgboundexpr} holds, we get the first part of Theorem~\ref{complexs}. Applying the result of Lemma~\ref{SvsU} yields the second part of the Theorem. 
\end{proof}

By this Theorem, we have retrieved the standard complexity in $\mathcal{O}\left(\epsilon^{-\sfrac{p+1}{p}}\right)$ (up to a logarithmic term) of standard adaptive regularization  that utilizes first up to $p$-th order derivative  to reach an $\epsilon$ first-order stationary point, see \cite{BirgGardMartSantToin17,Cartis2022-wb}.
The main contribution of the present analysis is that these complexity guarantees are obtained under substantially weaker smoothness assumptions that holds only locally and scales with the magnitude of the gradient, see \eqref{Lippth}.
For $p=1$, we retrieve the known complexity result (again, up to logarithmic term) under the $(L_0,L_1)$ smoothness that
have already been derived in \cite{Vankovetal25,Zhang2020Why} for nonconvex optimization. 
We  now specialize our results for the case of interest when $p=2$.
\begin{corr}{}
Suppose that Assumption~\ref{assum1}--\ref{assum4} hold and let $p=2$. Then the \texttt{AR$2$-LS} algorithm requires at most

\begin{align*}
	\left(1 + \frac{|\log (\gamma_1)|}{\log (\gamma_2)}\right) &\left[ \left(1 + \frac{\log(\kap{up}) |\log(\epsilon)|}{\log(2)} \right) \kappa_{\star,2}\epsilon^{-(p+1)/p} + \frac{|\log(\epsilon)| + \log(\|g_0\|)}{\log(2)} + 1\right] \\ 
	&+ \frac{1}{\log \gamma_2} \log\left(\frac{\sigma_{\max}}{\sigma_0}\right)
\end{align*}
iterations to produce a vector $x_\epsilon$ such that $\|\nabla_x^1 f(x_\epsilon)\| \leq \epsilon$, where $\kappa_{\star,2}$ is defined as
\begin{equation}\label{kapstardef2}
	\kappa_{\star,2} \eqdef  \frac{12\sqrt{2}\kap{sup}^\sfrac{3}{2}\sigma_{\max}^\sfrac{1}{2} (f(x_0) - f_{\rm low})}{\eta_1 \varsigma },
\end{equation}
with $\sigma_{\max}$ as in \eqref{exprsigmakbound}, $\kap{up}$ as in \eqref{kapupdef} and $\kap{sup}$ as defined  in the initialization of \texttt{ARp-LS}.
\end{corr}
When applied to $p=2$, we retrieve the results developed in \cite{GraJerToin25} where a method combining regularized Newton method and negative curvature has been proposed for function that verify the local  Lipschitz smoothness stated in Assumption~\ref{assum3} for $p=2$. In contrast, the algorithm is simpler and relies solely on  the properties of the cubic step to ensure convergence. We turn to the convex case.   
\subsection{Convex case}\label{convexanalysis}
We know begin by adding the additional assumption on the convexity function and stating some standard notation related to convex analysis, see \cite{Mishchenko2023,CartisGoulToint2012}.
\begin{assumption}\label{assum5}
	$f$ is a convex function.
\end{assumption}   
We also assume that the level sets are bounded, namely,
\begin{assumption}\label{assum6}
	\[
	\|x-x_\star\| \leq D \text{ for all $x$ such that } f(x) \leq f(x_0).
	\]
\end{assumption}
Let $f_\star$ be the (global) minimum of $f$ and 
\begin{equation}\label{Deltakdef}
\Delta_k = f(x_k) - f_\star.
\end{equation} 
In the following, we will derive a bound on the number of iterations to reach an iterate $x_\epsilon$ such that $\Delta_k \leq \epsilon$.  We now provide a Lemma that bounds $\Delta_k$ w.r.t $\|g_k\|$.
\begin{lem}{}
Suppose that Assumption~\ref{assum5} and \ref{assum6} hold. Then,
\begin{equation}\label{Deltakgk}
	\Delta_k \leq  D \|g_k\|.
\end{equation}
\end{lem}
\begin{proof}
Since $f(x_\star) \geq f(x_k) + \langle g_k, x_\star - x_k \rangle$ follows directly from convexity, rearranging the latter, using Cauchy-Schwartz and  Assumption~\ref{assum6} yields the desired result.
\end{proof}
We will keep all the division of iterations introduced in \eqref{calIipdown} and \eqref{calIdecr}.
Observe that all the previous lemmas proved in Subsection~\ref{nonconvexanalysis} remain  valid up to Lemma~\ref{Skgbound} since the latter use the criteria $\|g_k\| \geq \epsilon$. We now prove a new version that takes into account $\Delta_k \geq \epsilon$ before termination.

\begin{lem}{Skgboundconvex}
	Suppose that Assumption~\ref{assum1}, Assumption~\ref{assum3} and Assumption~\ref{assum4}--Assumption~\ref{assum6} hold. Then, we have that 
	\begin{equation}\label{Skgboundexprconvex}
		|\calS_k^{g\searrow}| \leq \log(\frac{D\kap{up}}{\epsilon}) \frac{|\calS_k^{decr}|}{\log(2)} + \frac{|\log(\epsilon)| +  \log(D) + \log(\|g_0\|)}{\log(2)} + 1,
	\end{equation} 
	with $\kap{up}$  defined in \eqref{kapupdef},
 $\kap{upgrad}$ defined at the initialization of \texttt{ARp-LS} and $\mu$ as in \eqref{mudef}.
	\end{lem}
	\begin{proof}
As the proof is similar to Lemma~\ref{Skgbound} and differs only by utilizing that $\Delta_k \geq \epsilon$, it is deferred to the Appendix~\ref{thirdmsproof}.
	\end{proof}
Equipped with the last lemmas, we are now in a position to state our new complexity theorem. 

\begin{theo}{complexsconvex}
	Suppose that Assumption~\ref{assum1}--\ref{assum6} hold. Then the \texttt{ARp-LS} algorithm requires at most
	\[
	|\calS_k| \leq \left(1 + \frac{\log(D\kap{up}) |\log(\epsilon)|}{\log(2)} \right) \kappa_{\star,convex}\epsilon^{-1/p}  + \frac{|\log(\epsilon)| + \log(D) + \log(\|g_0\|)}{\log(2)} + 1
	\]
	successful iterations and at most 
	\begin{align*}
		\left(1 + \frac{|\log (\gamma_1)|}{\log (\gamma_2)}\right) &\left[ \left(1 + \frac{\log({D\kap{up}}) |\log(\epsilon)|}{\log(2)} \right) \kappa_{\star,convex} \epsilon^{-1/p} + \frac{|\log(\epsilon)| + \log(D) + \log(\|g_0\|)}{\log(2)} + 1\right] \\ 
		&+ \frac{1}{\log \gamma_2} \log\left(\frac{\sigma_{\max}}{\sigma_0}\right)
	\end{align*}
	iterations to produce a vector $x_\epsilon$ such that $ f(x_k) - f_\star \leq \epsilon$, where $\kappa_{\star,convex}$ is defined as
	\begin{equation}\label{kapstarconvexdef}
		\kappa_{\star,convex} \eqdef  \frac{(p+1)! (f(x_0) - f_\star) (2\kap{sup})^\sfrac{p+1}{p} \sigma_{\max}^\sfrac{1}{p} + p\eta_1 \varsigma (f(x_0) - f_\star)^\sfrac{1}{p}+\eta_1}{\eta_1 \varsigma},
	\end{equation}
	with $\sigma_{\max}$ as in \eqref{exprsigmakbound}, $\kap{up}$ as in \eqref{kapupdef} and $\kap{sup}$ as defined  in the initialization of \texttt{ARp-LS}.
\end{theo}

\begin{proof}
Let us introduce a further subdivision in $\calS_k^{decr}$. Denote by 
\begin{equation}\label{largesmallgrad}
\calS_k^{decr,L}  \eqdef \left\{ i \in \calS_k^{decr} \Big| \|g_i\| \geq \varsigma \right\} \quad \calS_k^{decr,S} \eqdef \calS_k^{decr} \setminus \calS_k^{decr,L}. 
\end{equation}
Consider $i \in \calS_k^{decr,L}$. Using \eqref{toreuse} and that $\|g_i\| \geq \varsigma$ from \eqref{largesmallgrad} and \eqref{exprsigmakbound}, we obtain that
\[
f(x_0) - f_\star \geq \sum_{i=0}^k f(x_i) - f(x_{i+1}) \geq \sum_{i \in \calS_k^{decr,L}} \frac{\eta_1 \varsigma}{(p+1)!(2\kap{sup})^\sfrac{p+1}{p} \sigma_{i}^\sfrac{1}{p}}
\geq \frac{\eta_1 \varsigma |\calS_k^{decr,L}|}{(p+1)!(2\kap{sup})^\sfrac{p+1}{p} \sigma_{\max}^\sfrac{1}{p}},
\]
rearranging the last inequality gives that
\begin{equation}\label{SkdecrL}
|\calS_k^{decr,L}| \leq \frac{(p+1)! (f(x_0) - f_\star) (2\kap{sup})^\sfrac{p+1}{p} \sigma_{\max}^\sfrac{1}{p}}{\eta_1 \varsigma}.
\end{equation}
Now we focus on $|\calS_k^{decr,S}|$. Denote by $\calS_k^{decr,S} = \left\{j_t\right\}_{t=0}^{|\calS_k^{decr,S}|-1}$ where $j_t$ is an increasing sequence. We suppose that $|\calS_k^{decr,S}| \geq 2$. From \eqref{toreuse} and \eqref{largesmallgrad}, that the $f(x_i)$ sequence is non-increasing and that \eqref{Deltakgk} applies, we have that
\[
f(x_{j_t}) - f(x_{j_{t+1}}) \geq f(x_{j_t}) - f(x_{j_t+1}) \geq \frac{\eta_1 \varsigma \|g_{j_t}\|^\sfrac{p+1}{p}}{(p+1)!(2\varsigma\kap{sup})^\sfrac{p+1}{p}\sigma_{\max}^\sfrac{1}{p}} \geq \frac{\eta_1 (f(x_{j_t}) - f_\star)^\sfrac{p+1}{p}}{(p+1)!(2 D\kap{sup})^\sfrac{p+1}{p}(\varsigma\sigma_{\max})^\sfrac{1}{p}}.
\]
Defining $\alpha_t = f(x_{j_t}) - f_\star$ for $t \in \iibe{0}{|\calS_k^{decr,S}| -2}$, the last inequality can be rewritten as
\[
\alpha_t - \alpha_{t+1} \geq \frac{\eta_1 \alpha_t^\sfrac{p+1}{p}}{(p+1)!(2 D\kap{sup})^\sfrac{p+1}{p}(\varsigma\sigma_{\max})^\sfrac{1}{p}}.
\]
Denote $\kappa_{C} = \frac{\eta_1}{(p+1)!(2 D\kap{sup})^\sfrac{p+1}{p}(\varsigma\sigma_{\max})^\sfrac{1}{p}}$ and $\nu_t = \kappa_{C}^p \alpha_t$, we then obtain  that
\[
\nu_{t} - \nu_{t+1} \geq \nu_{t}^\sfrac{p+1}{p}.
\]
Proceeding exactly as in
 \cite[Theorem~4]{Nesterov2019}, we get that,
\begin{equation}\label{nutbound}
\frac{1}{\nu_t} \geq \left(\frac{1}{\nu_0^\sfrac{1}{p}} + \frac{t}{p}\right)^p = \frac{(t+p)^p}{\nu_0 p^p}.
\end{equation}
Upper-bounding $\alpha_t$ from the last inequality and that $\alpha_0 = \kappa_{C}^{-p} \nu_0$, we derive that
\begin{align*}
	\alpha_t = \kappa_C^{-p} \nu_t \leq \kappa_{C}^{-p} \frac{\nu_0 p^p}{(t+p)^p} = \frac{\alpha_0 p^p}{(t+p)^p}.
\end{align*}
Therefore, since before termination $\alpha_{|\calS_k^{decr,S}|-2} \geq \epsilon$ and $\alpha_0 \leq f(x_0) - f_\star$, we derive that
\begin{equation}\label{SkdecrS}
	|\calS_k^{decr,S}| \leq \frac{p (f(x_0) - f_\star)^\sfrac{1}{p} }{ \epsilon^\sfrac{1}{p}} + 1.
\end{equation}
Note that this bound still applies when $	|\calS_k^{decr,S}| \leq 1$.
Combining both \eqref{SkdecrL} and \eqref{SkdecrS}, we derive that
\begin{equation}\label{Skdecrbound}
|\calS_k^{decr}| = |\calS_k^{decr,L}| + |\calS_k^{decr,S}| \leq \frac{(p+1)! (f(x_0) - f_\star) (2\kap{sup})^\sfrac{p+1}{p} (\sigma_{\max})^\sfrac{1}{p} + p\eta_1 \varsigma (f(x_0) - f_\star)^\sfrac{1}{p}}{\eta_1 \varsigma \epsilon^{\sfrac{1}{p}}}  +1. 
\end{equation}
Using now the  definition of $\kappa_{\star,convex}$ in \eqref{kapstarconvexdef}, that $|\calS_k| = |\calS_k^{decr}| + |\calS_k^{g\searrow}|$ and the bound \eqref{Skgboundexprconvex} with the inequlity \eqref{Skdecrbound} yields the first part of Theorem~\ref{complexsconvex}. Using now the result of Lemma~\ref{SvsU} yields the second part of the Theorem.

\end{proof}

Theorem~3.2 therefore recovers the classical complexity bound for standard (non-accelerated) tensor methods applied to convex optimization problems~\cite{Nesterov2019}, while requiring substantially weaker smoothness assumptions.
 Under global Lipschitz continuity, sufficiently large values of the regularization parameter guarantee that the model $m_k$ is convex, making its minimization computationally tractable (see \cite{Nesterov2019}). That argument relies fundamentally on global smoothness of the $p$-th order tensor. In the present setting, establishing an analogous result would additionally require explicit knowledge of the locality parameter $\delta$, leading to a non-adaptive algorithm that would no longer be consistent with the philosophy of Algorithm~\ref{ARp-LS}.

\section{Numerical Illustration}\label{numeric-s}
In this section, we provide numerical illustration of our method for three different nonconvex regression tasks for the case $p=2$. The objective is to highlight the performance of \texttt{AR2-LS} on several nonconvex regression tasks and compare it to vanilla \texttt{AR2} (standard cubic regulaization $\frac{\sigma_{k}}{6}\|s\|^3$). As baselines, we will also consider other competitive fast regularized Newton method \texttt{AN2CLS} and \texttt{AN2C} developed in both \cite{GraJerToin24,GraJerToin25}.  Throughout
this section, $ \{a_i, y_i\}_{i=1}^m$ is the training data.   $a_i \in \mathbb{R}^n $ is the ith feature vector,
	$b_i$ denotes either the binary class label (
	$0$ or $1$) or the regression target, depending on the problem under consideration.

\noindent
For the first task, we consider the robust biweight Tucker regression problem \cite{Beaton1974}
\begin{equation}\label{beaton}
f(x) = \frac{1}{m} \sum_{i=1}^{m} \varphi(a_i^T x - b_i) \quad \text{ where } \varphi(\theta) = \frac{\theta^2}{1 + \theta^2}.
\end{equation} 

\noindent
The function $\varphi$ is a robust alternative of the quadratic loss; since it is considerably less sensitive to large values. To further increase the degree of nonconvexity, we follow the implementation detailed in \cite{pmlr-v70-carmon17a} for $a_i$ and $b_i$. The optimization process is initialized at $x_0=0$ and we consider $n=30$ and $m=30$. The results presented will be for 5000 independent instances of \eqref{beaton}
 
Next, we consider a nonconvex-binary classification problem as done in \cite{cartis2025efficient,GraJerToin24}
\begin{equation}\label{nonconvexbinreg}
	f(x) = \frac{1}{m} \sum_{i=1}^{m} \left( \frac{1}{1+\exp(-a_i^\intercal x)} - b_i \right)^2
\end{equation}
here $a_i \sim \mathcal{N}(0,4I_n)$ and $b_i$ is set to zero or one with equal probability and we initialize the optimization at $x_0 = 0$. We consider here $n=50$, $m=200$ and  5000 independently generated problem instances of \eqref{nonconvexbinreg}.

The final test problem is the phase retrieval problem \cite{candes2015phase}
 where the objective function is given by
\begin{equation}\label{phaseretrieval}
f(x) = \frac{1}{2m} \sum_{i=1}^{m} \left( (a_i^T x)^2 - b_i \right)^2.
\end{equation}
In this experiment, we set $n = m = 50$ with  $a_i \sim  \mathcal{N}(0, I_n)$,
a ground truth signal $x^\star \sim \mathcal{N}(0, 4I_n)$, then generate noisy observations 
$b = (Ax^*)^2 + 3\nu_1 + \nu_2$ where $\nu_1 \sim \mathcal{N}(0, I_m)$ is  and 
$\nu_2 \sim \{\mathcal{B}(0.3)\}_{i=1}^m$ as done in \cite{pmlr-v70-carmon17a}. The optimization process is initialized at 
$x_0   \sim \mathcal{N}(0, 4I_n)$, to have the same norm as the true solution $x^\star$. Note that for this example, Since the objective function \eqref{phaseretrieval} is quartic, the Hessian of \eqref{phaseretrieval} is not Lipschitz continuous and Assumption~\ref{assum3} is therefore more appropriate for this case.

\subsection{Results}\label{subsec-res}
We compare our method \texttt{AR2-LS} with an equivalent \texttt{AR2} \cite{Cartis2022-wb}.
To isolate the effect of the proposed regularization strategy~\eqref{model}, both algorithms employ the same set of algorithmic parameters whenever possible, and the cubic subproblem is solved exactly (i.e., $\theta_1=1$). The common hyperparameters are chosen as
\[
\sigma_{\min}=10^{-8}, \qquad
\gamma_1=0.5, \qquad
\gamma_2=10, \qquad
\gamma_3=10, \qquad
\eta_1=10^{-4}, \qquad
\eta_2=0.95, \qquad
\theta_1=1.
\]

For \texttt{AR2}, we initialize the regularization parameter with $\sigma_0=1$. For \texttt{AR2-LS}, we instead use the adaptive initialization
\[
\sigma_0=\frac{1}{\|g_0\|}, \qquad
\vartheta=\kappa_{\theta}=10^{5}.
\]

We also include \texttt{AN2C} and \texttt{AN2CLS} as competing baseline methods. These second-order algorithms alternate between regularized Newton steps and negative-curvature directions. The \texttt{AN2C} method was developed under the assumption of a globally Lipschitz continuous Hessian~\cite{GraJerToin24}, whereas \texttt{AN2CLS} was specifically designed for local Hessian Lipschitz continuity \cite{GraJerToin25} (namely, Assumption~\ref{assum3} with $p=2$). For both methods, we compute the minimum eigenvalue exactly and solve the regularized Newton sub-problem exactly. Their implementations follow the parameter choices reported in~\cite[Section~5]{GraJerToin25}.

All experiments were conducted in Julia on a machine equipped with an AMD Ryzen~7~5000 processor running at 3.8\,GHz. The optimization process is terminated once the gradient norm satisfies
\[
\|\nabla f(x_k)\|\leq \epsilon,
\]
with $\epsilon=10^{-6}$. As a performance metric, we report the number of iterations required for convergence.

To compare between the four proposed methods, we give plots that details the percentage of successfully solved instances (y-axis) for a budget of iteration (x-axis).

\begin{figure}[htbp]
	\centering

	\begin{subfigure}{0.45\textwidth}
		\centering
		\includegraphics[width=\linewidth]{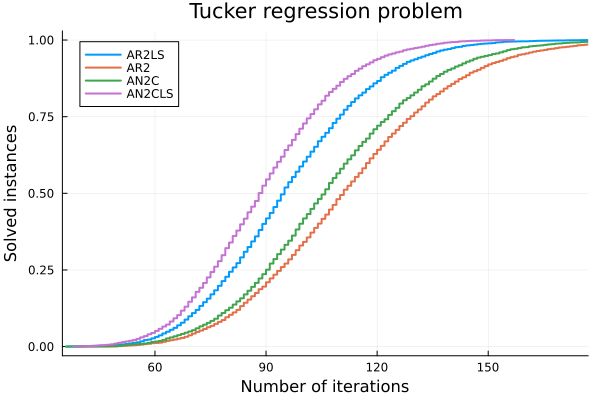} 
		\caption{Tucker regression problem as specified in \eqref{beaton}}
		\label{fig:sub1}
	\end{subfigure}
	\hfill  
	\begin{subfigure}{0.45\textwidth}
		\centering
		\includegraphics[width=\linewidth]{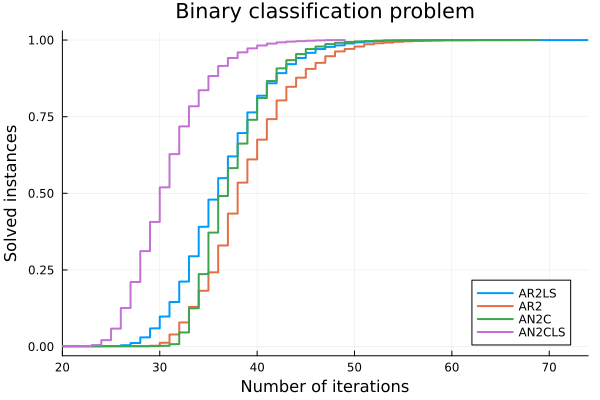} 
		\caption{Binary nonconvex classification as specified in \eqref{nonconvexbinreg}}
		\label{fig:sub2}
	\end{subfigure}
	
	\vspace{1em} % vertical space between rows
	
	% Second row: one subfigure centered
	\begin{subfigure}{0.5\textwidth}  % you can adjust width
		\centering
		\includegraphics[width=\linewidth]{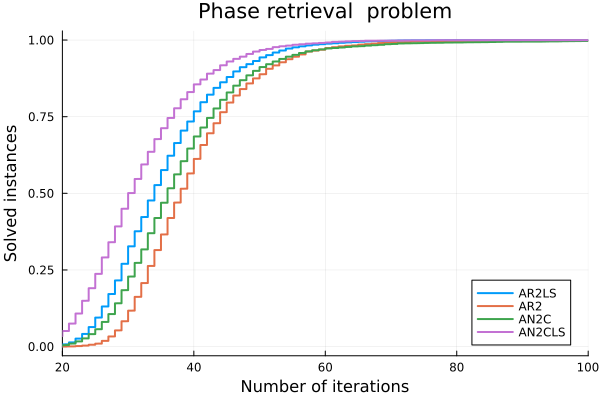} 
		\caption{Phase retrieval problem as specified in \eqref{phaseretrieval}}
		\label{fig:sub3}
	\end{subfigure}
	
	\caption{Comparison of performance between \texttt{AR2LS}, \texttt{AR2}, \texttt{AN2C} and \texttt{AN2CLS} for three different nonconvex regression tasks.}
	\label{fig:main}
\end{figure}  

Figure~\ref{fig:main} shows that \texttt{AR2-LS} consistently outperforms the classical \texttt{AR2} algorithm across all three benchmark problems. In particular, for the phase retrieval problem~\eqref{phaseretrieval}, whose objective function is quartic, methods designed to exploit local smoothness (\texttt{AR2-LS} and \texttt{AN2CLS}) are consistently more efficient than those developed under the standard global Lipschitz smoothness assumption (\texttt{AR2} and \texttt{AN2C}).
Overall, \texttt{AR2-LS} comes in second place, behaves in most case better than \texttt{AN2C} but trail behind \texttt{AN2CLS}. This may be explained by the fact that the regularization of the Newton system proportional to $\|g_k\|$ used by \texttt{AN2CLS} was shown to be optimal for a wide range of convex problem, see \cite{DoikovMischNes24,Doikov2025} and the references therein.

These numerical experiments are preliminary and primarily intended to illustrate the practical potential of the proposed \texttt{AR2-LS} method. A more comprehensive numerical study would require a careful tuning of the algorithmic parameters \footnote{For instance, one could consider an iteration-dependent adaptive choice of $\varsigma$ in Algorithm~\ref{ARp-LS}}. Additional nonlinear optimization benchmarks should also be included to better assess the practical performance and robustness of \texttt{AR2-LS}.

\section{Conclusion and Perspectives}\label{concl-s}
In the current paper, we have proposed an adaptive tensor method that handles functions that have locally Lipschitz smooth $p$-th order tensor. The new smoothness condition is assumed to hold locally around a ball for any given iterate and takes also into account the magnitude of the gradient. At variance with standard Lipschitz condition, this condition covers a boarder class of problems such as univariate polynomial and exponential. The algorithm resembles standard adaptive regularization methods \cite{BirgGardMartSantToin17,Nesterov2019} but differs by adjusting the expression of the regularization parameter. In addition to the usual decrease ratio, we also introduce additional tests to take into account the fact that the Lipschitz condition is local. Initial experimental results highlight some merits of this approach for some regression tasks.

In this line of work, a natural extension would be to introduce broader smoothness assumptions related to Assumption~\ref{assum3} as it was previously done for $(L_0, L_1)$ smoothness condition \cite{Chenetal23,ReiszeiJadb25,ZhangFang20}. Furthermore, the proposed regularization parameter formula \eqref{model} could be integrated into the recently developed efficient third-order adaptive tensor methods \cite{cartis2025efficient}, offering a promising avenue for future research.

\section*{Data Availability Statement}  
The synthetic data used in this study were generated entirely by random processes. 
No real-world or publicly available datasets were used.

\section*{Acknowledgments}
 This work was supported
by the Hong Kong Innovation and Technology Commission (InnoHK Project CIMDA). The author would like to thank Prof.\ Cartis for valuable advice that helped improve this work.

{ \footnotesize
	\bibliographystyle{plain}
	\bibliography{refs}

@BOOK{Cartis2022-wb,
	title     = {Evaluation Complexity of Algorithms for Nonconvex Optimization},
	author    = {Cartis, Coralia and Gould, Nicholas I. M. and Toint, Philippe L.},
	publisher = {Society for Industrial and Applied Mathematics},
	series    = {MOS-SIAM Series on Optimization},
	year      = 2022,
	address   = {Philadelphia, PA}
}

@inproceedings{Zhang2020Why,
	title     = {Why Gradient Clipping Accelerates Training: A Theoretical Justification for Adaptivity},
	author    = {Zhang, Jingzhao and He, Tianxing and Sra, Suvrit and Jadbabaie, Ali},
	booktitle = {International Conference on Learning Representations},
	year      = 2020
}

@article{BirgGardMartSantToin17,
	title   = {Worst-Case Evaluation Complexity for Unconstrained Nonlinear Optimization Using High-Order Regularized Models},
	author  = {Birgin, E. G. and Gardenghi, J. L. and Mart{\'i}nez, J. M. and Santos, S. A. and Toint, Ph. L.},
	journal = {Mathematical Programming},
	volume  = {163},
	number  = {1-2},
	pages   = {359--368},
	year    = 2017
}

@article{Nesterov2019,
	title   = {Implementable Tensor Methods in Unconstrained Convex Optimization},
	author  = {Nesterov, Yurii},
	journal = {Mathematical Programming},
	volume  = {186},
	number  = {1-2},
	pages   = {157--183},
	year    = 2019
}

@article{CartGoulToin20b,
	title   = {Sharp Worst-Case Evaluation Complexity Bounds for Arbitrary-Order Nonconvex Optimization with Inexpensive Constraints},
	author  = {Cartis, Coralia and Gould, Nicholas I. M. and Toint, Philippe L.},
	journal = {SIAM Journal on Optimization},
	volume  = {30},
	number  = {1},
	pages   = {513--541},
	year    = 2020
}

@inproceedings{Chenetal23,
	title     = {Generalized-Smooth Nonconvex Optimization Is as Efficient as Smooth Nonconvex Optimization},
	author    = {Chen, Ziyi and Zhou, Yi and Liang, Yingbin and Lu, Zhaosong},
	booktitle = {Proceedings of the 40th International Conference on Machine Learning},
	series    = {Proceedings of Machine Learning Research},
	articleno = {214},
	numpages  = {32},
	year      = 2023
}

@inproceedings{Vankovetal25,
	title     = {Optimizing $(L_0, L_1)$-Smooth Functions by Gradient Methods},
	author    = {Vankov, Daniil and Rodomanov, Anton and Nedich, Angelia and Sankar, Lalitha and Stich, Sebastian U.},
	booktitle = {The Thirteenth International Conference on Learning Representations},
	year      = 2025
}

@inproceedings{Koloskova23,
	title     = {Revisiting Gradient Clipping: Stochastic Bias and Tight Convergence Guarantees},
	author    = {Koloskova, Anastasia and Hendrikx, Hadrien and Stich, Sebastian U.},
	booktitle = {Proceedings of the 40th International Conference on Machine Learning},
	series    = {Proceedings of Machine Learning Research},
	articleno = {714},
	numpages  = {21},
	year      = 2023
}

@inproceedings{Leietal23,
	title     = {Convex and Non-Convex Optimization under Generalized Smoothness},
	author    = {Li, Haochuan and Qian, Jian and Tian, Yi and Rakhlin, Alexander and Jadbabaie, Ali},
	booktitle = {Advances in Neural Information Processing Systems},
	volume    = {36},
	articleno = {1749},
	numpages  = {34},
	year      = 2023
}

@InProceedings{hubler24a,
	title     = {Parameter-Agnostic Optimization under Relaxed Smoothness},
	author    = {H{\"u}bler, Florian and Yang, Junchi and Li, Xiang and He, Niao},
	booktitle = {Proceedings of The 27th International Conference on Artificial Intelligence and Statistics},
	series    = {Proceedings of Machine Learning Research},
	pages     = {4861--4869},
	year      = 2024
}

@article{GratToin21,
	title   = {Adaptive Regularization Minimization Algorithms with Nonsmooth Norms},
	author  = {Gratton, S. and Toint, Philippe L.},
	journal = {IMA Journal of Numerical Analysis},
	volume  = {43},
	number  = {3},
	pages   = {1313--1340},
	year    = 2023
}

@article{Mishchenko2023,
	title   = {Regularized {Newton} Method with Global $O(1/k^2)$ Convergence},
	author  = {Mishchenko, Konstantin},
	journal = {SIAM Journal on Optimization},
	volume  = {33},
	number  = {3},
	pages   = {1440--1462},
	year    = 2023
}

@article{GraJerToin24,
	title   = {Yet Another Fast Variant of {Newton's} Method for Nonconvex Optimization},
	author  = {Gratton, Serge and Jerad, Sadok and Toint, Philippe L.},
	journal = {IMA Journal of Numerical Analysis},
	year    = 2024,
		volume  = {45},
	number  = {2},
		pages   = {971--1008},
			year    = 2025
}

@article{OFFO-ARp,
	title   = {Convergence Properties of an Objective-Function-Free Optimization Regularization Algorithm, Including an $\mathcal{O}(\epsilon^{-3/2})$ Complexity Bound},
	author  = {Gratton, Serge and Jerad, Sadok and Toint, Philippe L.},
	journal = {SIAM Journal on Optimization},
	volume  = {33},
	number  = {3},
	pages   = {1621--1646},
	year    = 2023
}

@article{CartGoulToin19,
	title   = {Universal Regularization Methods: Varying the Power, the Smoothness and the Accuracy},
	author  = {Cartis, Coralia and Gould, Nicholas I. M. and Toint, Philippe L.},
	journal = {SIAM Journal on Optimization},
	volume  = {29},
	number  = {1},
	pages   = {595--615},
	year    = 2019
}

@book{Yap99,
	title     = {Fundamental Problems of Algorithmic Algebra},
	author    = {Yap, Chee Keng},
	publisher = {Oxford University Press},
	address   = {New York, NY},
	year      = 1999
}

@article{GraJerToin25,
	title   = {A Fast {Newton} Method under Local {Lipschitz} Smoothness},
	author  = {Gratton, Serge and Jerad, Sadok and Toint, Philippe},
	journal = {EURO Journal on Computational Optimization},
	volume  = {14},
	pages   = {100--128},
	year    = 2026
}

@article{CartisGoulToint2012,
	title   = {Evaluation Complexity of Adaptive Cubic Regularization Methods for Convex Unconstrained Optimization},
	author  = {Cartis, Coralia and Gould, Nicholas I. M. and Toint, Philippe L.},
	journal = {Optimization Methods and Software},
	volume  = {27},
	number  = {2},
	pages   = {197--219},
	year    = 2012
}

@article{Xieyinu24,
	title   = {Trust Region Methods for Nonconvex Stochastic Optimization beyond {Lipschitz} Smoothness},
	author  = {Xie, Chenghan and Li, Chenxi and Zhang, Chuwen and Deng, Qi and Ge, Dongdong and Ye, Yinyu},
	journal = {Proceedings of the AAAI Conference on Artificial Intelligence},
	volume  = {38},
	number  = {14},
	pages   = {16049--16057},
	year    = 2024
}

@article{ZhouMa25,
	title   = {{AdaBB}: Adaptive {Barzilai-Borwein} Method for Convex Optimization},
	author  = {Zhou, Danqing and Ma, Shiqian and Yang, Junfeng},
	journal = {Mathematics of Operations Research},
	  volume = {51},
	number = {1},
	year = {2026},
	pages = {715–745}
}

@article{Bauschke2017,
	title   = {A Descent Lemma Beyond {Lipschitz} Gradient Continuity: First-Order Methods Revisited and Applications},
	author  = {Bauschke, Heinz H. and Bolte, J{\'e}r{\^o}me and Teboulle, Marc},
	journal = {Mathematics of Operations Research},
	volume  = {42},
	number  = {2},
	pages   = {330--348},
	year    = 2017
}

@article{Lu2018,
	title   = {Relatively Smooth Convex Optimization by First-Order Methods, and Applications},
	author  = {Lu, Haihao and Freund, Robert M. and Nesterov, Yurii},
	journal = {SIAM Journal on Optimization},
	volume  = {28},
	number  = {1},
	pages   = {333--354},
	year    = 2018
}

@inproceedings{MalMischenko24,
	title     = {Adaptive Proximal Gradient Method for Convex Optimization},
	author    = {Malitsky, Yura and Mishchenko, Konstantin},
	booktitle = {Advances in Neural Information Processing Systems},
	volume    = {37},
	articleno = {3193},
	numpages  = {28},
	year      = 2024
}

@inproceedings{ZhangFang20,
	title     = {Improved Analysis of Clipping Algorithms for Non-Convex Optimization},
	author    = {Zhang, Bohang and Jin, Jikai and Fang, Cong and Wang, Liwei},
	booktitle = {Advances in Neural Information Processing Systems},
	volume    = {33},
	articleno = {1301},
	numpages  = {11},
	year      = 2020
}

@InProceedings{pmlr-v206-sun23d,
	title     = {Convergence of {Stein} Variational Gradient Descent under a Weaker Smoothness Condition},
	author    = {Sun, Lukang and Karagulyan, Avetik and Richtarik, Peter},
	booktitle = {Proceedings of The 26th International Conference on Artificial Intelligence and Statistics},
	series    = {Proceedings of Machine Learning Research},
	pages     = {3693--3717},
	year      = 2023,
	volume    = {206}
}

@InProceedings{pmlr-v195-faw23a,
	title     = {Beyond Uniform Smoothness: A Stopped Analysis of Adaptive {SGD}},
	author    = {Faw, Matthew and Rout, Litu and Caramanis, Constantine and Shakkottai, Sanjay},
	booktitle = {Proceedings of Thirty Sixth Conference on Learning Theory},
	series    = {Proceedings of Machine Learning Research},
	pages     = {89--160},
	year      = 2023,
	volume    = {195}
}

@INPROCEEDINGS{ReiszeiJadb25,
	title     = {Variance-Reduced Clipping for Non-Convex Optimization},
	author    = {Reisizadeh, Amirhossein and Li, Haochuan and Das, Subhro and Jadbabaie, Ali},
	booktitle = {ICASSP 2025 - 2025 IEEE International Conference on Acoustics, Speech and Signal Processing (ICASSP)},
	year      = 2025
}

@article{Doikov2021,
	title   = {Local Convergence of Tensor Methods},
	author  = {Doikov, Nikita and Nesterov, Yurii},
	journal = {Mathematical Programming},
	volume  = {193},
	number  = {1},
	pages   = {315--336},
	year    = 2022
}

@inproceedings{Semenovetal25,
	title     = {Gradient-Normalized Smoothness for Optimization with Approximate Hessians},
	author    = {Semenov, Andrei and Jaggi, Martin and Doikov, Nikita},
	booktitle = {The Fourteenth International Conference on Learning Representations},
	year      = 2026
}

@article{DoikovMischNes24,
	title   = {Super-Universal Regularized {Newton} Method},
	author  = {Doikov, Nikita and Mishchenko, Konstantin and Nesterov, Yurii},
	journal = {SIAM Journal on Optimization},
	volume  = {34},
	number  = {1},
	pages   = {27--56},
	year    = 2024
}

@article{Doikov2025,
	title   = {Minimizing Quasi-Self-Concordant Functions by Gradient Regularization of {Newton} Method},
	author  = {Doikov, Nikita},
	journal = {Mathematical Programming},
	year    = 2025,
}

@article{cartis2025efficient,
	title   = {Efficient Implementation of Third-Order Tensor Methods with Adaptive Regularization for Unconstrained Optimization},
	author  = {Cartis, Coralia and Hauser, Raphael and Liu, Yang and Welzel, Karl and Zhu, Wenqi},
	journal = {Mathematical Programming Computation},
	year    = 2026
}

@article{Beaton1974,
	title   = {The Fitting of Power Series, Meaning Polynomials, Illustrated on Band-Spectroscopic Data},
	author  = {Beaton, Albert E. and Tukey, John W.},
	journal = {Technometrics},
	volume  = {16},
	number  = {2},
	pages   = {147--185},
	year    = 1974
}

@InProceedings{pmlr-v70-carmon17a,
	title     = {``{Convex} Until Proven Guilty'': Dimension-Free Acceleration of Gradient Descent on Non-Convex Functions},
	author    = {Carmon, Yair and Duchi, John C. and Hinder, Oliver and Sidford, Aaron},
	booktitle = {Proceedings of the 34th International Conference on Machine Learning},
	series    = {Proceedings of Machine Learning Research},
	pages     = {654--663},
	year      = 2017,
	volume    = {70}
}

@article{candes2015phase,
	author = {Candès, Emmanuel J. and Li, Xiaodong and Soltanolkotabi, Mahdi},
	title = {Phase Retrieval via Wirtinger Flow: Theory and Algorithms},
	journal = {IEEE Transactions on Information Theory},
	volume = {61},
	number = {4},
	pages = {1985--2007},
	year = {2015},
	doi = {10.1109/TIT.2015.2399924}
}
}
\appendix
\section{Missing Proofs}
\subsection{Proof of Lemma~\ref{caractAS3}}\label{firstmsproof}
Let $h(t)$ be defined as $h(t) \eqdef \nabla_x^p f(x+t(y-x))$, $t \in [0,1]$ then $h^\prime(t) = \nabla_x^{p+1}f(x+t(y-x))[y-x]$, Then we have that from \eqref{pplusone} and \eqref{L0L1}
\begin{align}\label{appen1}
\|\nabla_x^p f(y) - \nabla_x^p f(x)\| &= \|h(1)  - h(0)\| \nonumber \\
&= \|\int_{0}^{1} \nabla_x^{p+1} f(x+t(y-x))[y-x] \, dt\| \nonumber \\
&\leq \int_{0}^1 \|\nabla_x^{p+1} \nonumber f(x+t(y-x))[y-x] \|  dt \\
&\leq \int_{0}^1 (M_0 + M_1 \|\nabla_x^1f(x+t(y-x))\|) \|y-x\| dt \nonumber \\
&= M_0 \|y-x\| + M_1 \|y-x\| \int_{0}^1  \|\nabla_x^1f(x+t(y-x))\|  dt 
\end{align}
We now move to provide a bound on $\|\nabla_x^1f(x+t(y-x))\|$ for $t \in [0,1]$.
Suppose that $\|x-y\| \leq \frac{1}{G_1}$. By using \eqref{L0L1}, \cite[Lemma~2.5]{Vankovetal25} and that $e^u \leq 1+2u$ for $u \in [0,1]$, we derive that
 \begin{align*}
\|\nabla_x^1f(x+t(y-x))\| &\leq \|\nabla_x^1f(x+t(y-x)) - \nabla_x^1 f(x)\|+  \|\nabla_x^1f(x)\| \\
&\leq \|\nabla_x^1f(x)\| + (G_0+G_1\|\nabla_x^1f(x)\|) \frac{e^{G_1 t \|(y-x)\|}-1}{G_1} \\
&\leq \|\nabla_x^1f(x)\| + 2G_0 t \|y-x\| + 2G_1  \|\nabla_x^1f(x)\| t\|y-x\| \\
&\leq  \|\nabla_x^1 f(x)\| (1+2t) +2 t\frac{G_0}{G_1},   
 \end{align*}
 where we used that $ \|y-x\| \leq  \frac{1}{G_1}$ to derive the last inequality.
 Now injecting the last inequality in \eqref{appen1} and evaluating the integral in $t$, we obtain for $\|x-y\| \leq \frac{1}{G_1}$,
 \[
 \|\nabla_x^p f(y) - \nabla_x^p f(x)\|  \leq M_0 \|y-x\| + 2M_1 \|y-x\| \|\nabla_x^1 f(x)\| +  \frac{M_1G_0}{G_1} \|y-x\|,  
 \] 
 which is the statement of Lemma~\ref{caractAS3}.
 \subsection{Proof of Lemma~\ref{approxfungrad}}
 Let $x,s \in \Ren$ with $\|s\| \leq \delta$. From \cite[Theorem~A.7.1]{Cartis2022-wb}, tensor inequalities and Assumption~\ref{assum3}, we derive that, 
 \begin{align*}
\left|f(x+s) - T_{f,p}(x,s)\right| &= \left|\frac{1}{(p-1)!} \int_{0}^1 (1-t)^{p-1} (\nabla_x^p f(x+ts) - \nabla_x^p f(x))[s]^p dt \right| \\
&\leq \frac{1}{(p-1)!} \int_{0}^1 (1-t)^{p-1}  \|\nabla_x^p f(x+ts) - \nabla_x^p f(x)\| \|s\|^p dt \\
&\leq \frac{\|s\|^{p+1}}{(p-1)!} \int_{0}^1 (L_0 + L_1 \|\nabla_x^1 f(x)\|) t  (1-t)^{p-1} dt = \frac{L_0 + L_1 \|\nabla_x^1 f(x)\| }{(p+1)!} \|s\|^{p+1}.
 \end{align*}
 We know move to the second part. Again, as in  \cite[Theorem~A.7.1]{Cartis2022-wb}, tensor inequalities and Assumption~\ref{assum3},
 \begin{align*}
 	\|\nabla_x^1 f(x+s) - \nabla_s^1 T_{f,p}(x,s)\| &\leq \frac{1}{(p-2)!} \int_{0}^1 (1-t)^{p-2} \|\nabla_x^p f(x+ts) - \nabla_x^p f(x)\| \|s\|^{p-1} dt \\
 	&\leq \frac{\|s\|^p}{(p-2)!} \int_{0}^1 (1-t)^{p-2} t (L_0 + L_1 \|\nabla_x^1 f(x)\|) dt \\ 
 	&= \frac{\|s\|^p}{p!} (L_0 + L_1 \|\nabla_x^1f(x)\|),
 \end{align*} 
 thus giving the second part of the Lemma.
\subsection{Proof of Lemma~\ref{Skgboundconvex}}\label{thirdmsproof}
\begin{proof}
First note that if $k \in \calS_k^{g \searrow}$, $\|g_{k+1}\| \leq \frac{\|g_k\|}{2}$. Let $k \in \calS_k^{decr}$. Using now that \eqref{testnextgrad} does not hold, that $ \frac{\epsilon}{D} \leq \|g_k\|$ since $\Delta_k \geq \epsilon$ and \eqref{Deltakgk} holds, $\max(\varsigma, \epsilon) \leq 1$, and that \eqref{exprskbound} applies, we derive that,
\begin{equation}\label{gkplusovergkconvex}
	\frac{\|g_{k+1}\|}{\|g_k\|} \leq \kap{upgrad} \max(\frac{\varsigma}{\|g_k\|}, 1) \sigma_{k} \|s_k\|^{p} \leq \kap{upgrad} \max(\frac{\varsigma D}{\epsilon}, 1) \sigma_{k} \|s_k\|^{p} \leq \frac{D\kap{upgrad}  \mu^{p}}{\epsilon} = \frac{D\kap{up}}{\epsilon},
\end{equation}
where $\kap{up}$ is defined in \eqref{kapupdef}.
Successively using that $\calS_k = \calS_k^{decr} \cap  \calS_k^{g \searrow}$, the bound on $\frac{\|g_{k+1}\|}{\|g_k\|}$ in both cases either $i \in \calS_k^{decr}$ \eqref{gkplusovergkconvex} or $i \in \calS_k^{g \searrow}$ and that $\|g_k\| \geq \frac{\epsilon}{D}$ before termination from \eqref{Deltakgk}, we obtain that
\begin{align*}
	\frac{\epsilon}{D\|g_0\|} &\leq \frac{\|g_k\|}{\|g_0\|} = \prod_{i \in \calS_k \setminus \{k\}} \frac{\|g_{i+1}\|}{\|g_i\|} = \prod_{i \in \calS_k^{decr} \setminus \{k\} } \frac{\|g_{i+1}\|}{\|g_i\|} \prod_{i \in \calS_k^{g\searrow} \setminus \{k\}} \frac{\|g_{i+1}\|}{\|g_i\|} \\
	&\leq \left(\frac{D\kap{up}}{\epsilon}\right)^{\calS_k^{decr} \setminus \{k\}} \left(\frac{1}{2}\right)^{ \calS_k^{g\searrow} \setminus \{k\}}.
\end{align*}
Rearranging the last inequality and using that $|\calS_k^{decr} \setminus \{k\}| \leq |\calS_k^{decr}|$ yields that
\[
\frac{2^{ \calS_k^{g\searrow} \setminus \{k\} } \epsilon}{D\|g_0\|} \leq \left(\frac{D\kap{up}}{\epsilon}\right)^{\calS_k^{decr}}.
\]
Taking the logarithm in the last inequality, using $|\calS_k^{g\searrow} \setminus \{k\}| \geq |\calS_k^{g\searrow}| - 1$ and further rearranging yields the stated result \eqref{Skgboundexprconvex}.
\end{proof}
\end{document}